\documentclass[11pt]{amsart}

\usepackage{amsmath,amssymb,amsthm}
\usepackage{graphicx}
\usepackage{url}
\usepackage{hyperref}

\newtheorem{theorem}{Theorem}
\newtheorem{corollary}[theorem]{Corollary}

\begin{document}

\title[Windowed Square Waves, $\tanh$, $\zeta(2k)$, and $L$-Values]{From a Windowed
Square Wave to the Hyperbolic Tangent, Even Zeta Values, and Dirichlet $L$-Values}

\author{Peter J. Bevelacqua, Ph.D. }
\email{pjbevel@gmail.com}

\maketitle

\begin{abstract}
  We window a delayed periodic square wave with a decaying signum function and evaluate the product at zero frequency in both the time and frequency domains. The window renders every series that appears absolutely convergent, so the argument requires only elementary integrals and a limit. We
  obtain new unified derivations for the Mittag-Leffler partial-fraction expansion for the hyperbolic tangent, the Fourier expansion 
  of the first Euler polynomial $E_1$, and the Basel sum $\zeta(2)$. The same calculation, followed by termwise integration, yields the Dirichlet beta value $\beta(3)=\pi^{3}/32$, and the values at $s=3$ of two other Dirichlet $L$-functions. Termwise integration of this sine series produces $\zeta(4)$, and an induction continues the process to $\zeta(2k)$ for every $k\ge1$. 
\end{abstract}

\section{Introduction}
In this article, we window a delayed periodic square wave with a decaying
signum function and compute the total integral of the product. This product is evaluated twice: once in the frequency domain and once
by elementary integration in the time domain. Equating the two answers
yields Theorem~\ref{thm:main}, a two-parameter cosine identity over the
odd integers. The identity itself is classical (it is formula~1.55,
p.~13, of Oberhettinger's tables \cite{Oberhettinger}), but the rest of
the article unfolds from it by specialization and termwise integration.
Freezing the delay gives the Mittag-Leffler partial-fraction expansion
of the hyperbolic tangent \cite{WhittakerWatson}, with no appeal to
complex analysis; sending the decay rate to zero instead gives the
Fourier expansion of the first Euler polynomial $E_1$
\cite{AbramowitzStegun,ApostolANT} and, at zero delay, the Basel sum
$\zeta(2)$. Section~3 then climbs an integral ladder, where termwise
integration divides the summand by $n$ and alternates the series between
sine and cosine. Evaluating the rungs of this ladder at the shifts
$d=\tfrac14,\tfrac16,$ and $\tfrac18$ produces the Dirichlet beta value
$\beta(3)=\pi^{3}/32$ and the values at $s=3$ of two further Dirichlet
$L$-functions attached to odd characters of modulus $6$ and $8$, while
$d=0$ produces $\zeta(4)$. The section closes with a short induction
(Theorem~\ref{thm:zeta2k}) extending the ladder to $\zeta(2k)$ for
every $k\ge 1$.

The original evaluation of $\zeta(2)$ and $\zeta(2k)$ goes back to Euler
\cite{Euler1740,Dunham1999}, and dozens of proofs have accumulated
since, catalogued and compared in the surveys
\cite{Chapman1999,Kalman1993,AignerZiegler}. Nearest to this article's
approach are the Fourier-analytic ones, in which $\zeta(2)$ falls out
of the Fourier series of a triangle wave \cite{SteinShakarchi}. Every
individual result obtained here is well known. What is believed to be
new is the route, where a single windowed signal whose total integral,
computed two ways, generates them all.

An important feature of the decaying signum window is that it
improves the decay of the spectral terms from $1/n$ to $1/n^{2}$, so
every series that appears converges absolutely. This article's argument uses
 elementary integrals, the classical pointwise convergence and
uniform boundedness of the partial sums of the square wave's Fourier
series, dominated convergence, and Maclaurin expansions, but never
Parseval's theorem or the $L^{2}$ theory of Fourier series. Throughout,
Fourier transforms are written in the engineering convention, in terms
of the ordinary frequency $f$.

\section{Windowed Square Wave in the Time and Frequency Domains}

In signal processing courses where functions of time and their corresponding 
representation in the frequency domain are studied, it is well known that the total 
integral of a function, 
$\int_{-\infty}^{\infty}x(t)\,dt$, is equivalent to the Fourier transform of the function evaluated at $f=0$. Our plan is to analyze a waveform by calculating this value in two
independent ways, once in the frequency domain and once by elementary integration 
in the time domain. 
We will then set them equal and see what we can learn. 

We first consider the zero-mean periodic square wave $s(t)$ defined by

\begin{equation}
s(t)=
\begin{cases}
+1, & 0 \leq t<\tfrac12,\\[2pt]
-1, & \tfrac12 \leq t<1,
\end{cases}
\qquad\qquad s(t+1)=s(t).
\label{eq:s} 
\end{equation}

Next we look at the signum function, with an exponential taper of
decay rate $a>0$,

\begin{equation}
\operatorname{sgn}(a, t)=
\begin{cases}
e^{-at}, & 0 \leq t ,\\[2pt]
-e^{at}, & t<0.
\end{cases}
\label{eq:sgn} 
\end{equation}

We study the product $p(a, d, t)=s(t-d)\operatorname{sgn}(a,t)$,
writing Fourier transforms in the engineering convention
$X(f)=\int_{-\infty}^{\infty}x(t)\,e^{-2\pi i f t}\,dt $. The original motivation behind 
the delay $d$ will be discussed in Section~3. The introduction of the parameter $d$, while seemingly 
of minor consequence, will turn into the key building block for the paper.
Figure~\ref{fig:construction} shows the two factors and their product.

\begin{figure}[htb]
\centering
\includegraphics[width=0.8\linewidth]{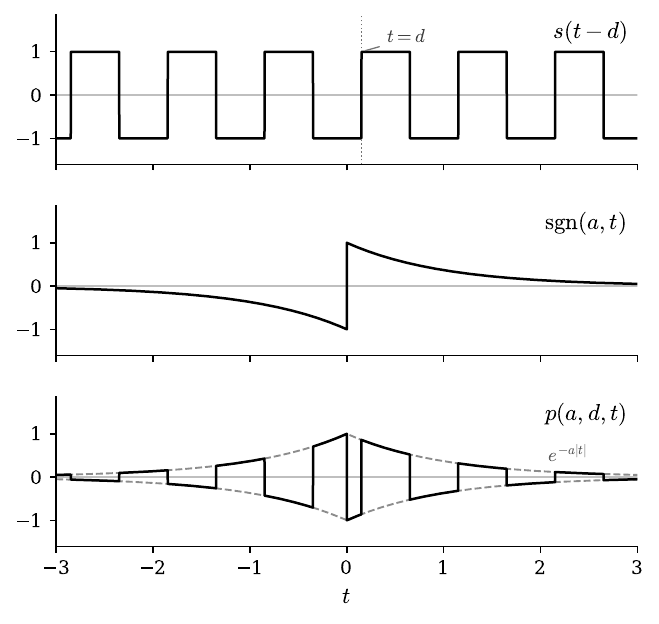}
\caption{The construction behind Theorem~\ref{thm:main}, drawn for
$a=1$ and $d=0.15$: the delayed square wave $s(t-d)$ (top), the
decaying signum window $\operatorname{sgn}(a,t)$ (middle), and their
product $p(a,d,t)$ (bottom), whose total integral---its value at zero
frequency---is computed twice. Dashed curves show the envelope
$\pm e^{-a|t|}$.}
\label{fig:construction}
\end{figure}

The multiplication of the two functions serves two purposes. First, 
it makes the total integral of the product nonzero for $d=0$. Second, it results in an 
absolutely convergent sum for all $f$, allowing for a rigorous calculation of the integral.

\subsection*{The frequency domain.}

We will obtain the frequency domain representations of $s(t)$ and sgn$(a,t)$. The Fourier series representation of the
square wave, $s(t)=\sum_{n}s_ne^{i2\pi n t}$, has Fourier coefficients
$s_n$ given by

\begin{equation}
s_n=
\begin{cases}
\dfrac{2}{i\pi n}, & n \text{ odd},\\[6pt]
0, & n \text{ even}.
\end{cases}
\label{eq:sqcoef}
\end{equation}

The Fourier transform of the decaying signum function can be directly
calculated:

\begin{equation}
\operatorname{SGN}(a, f)
= \int_{-\infty}^{0} (-e^{at}) e^{-i2\pi f t} \, dt + \int_{0}^{\infty} 
e^{-at} e^{-i2\pi f t} \, dt.
\label{eq:sqn_def}
\end{equation}

Equation \eqref{eq:sqn_def} can be simplified to

\begin{equation}
\operatorname{SGN}(a, f)= \dfrac{-4 \pi i f}{a^2+4\pi^2f^2}
\label{eq:sgn_f}
\end{equation}

To obtain the Fourier transform of the product, we will need to swap the order 
of integration and summation, and then integrate $s(t-d)$ term by term.
To see that this is mathematically sound, let

\begin{equation}
s_N(t)=\sum_{|n|\le N} s_n\,e^{i2\pi nt}
\qquad N=1,3,5,\dots
\label{eq:partial}
\end{equation}

denote the partial sums, so that

\begin{equation}
\int_{-\infty}^{\infty} s(t-d)\operatorname{sgn}(a,t)e^{-i2\pi f t}\,dt
=\int_{-\infty}^{\infty}\lim_{N\to\infty}\sum_{|n|\le N} e^{i2\pi n (t-d)}\, s_n \operatorname{sgn}(a,t)e^{-i2\pi f t}dt.
\label{eq:Pf}
\end{equation}

Two classical facts about the square wave allow us to simplify the derivation. Since
$s(t)$ is of bounded variation, the partial sums $s_N$ converge to $s$ at
each point of continuity (and to the midpoint value $0$ at the jumps). Further, the partial sums $s_N(t)$ are
uniformly bounded, so $|s_N(t)|\le C$ for all $N$ and $t$ \cite{Zygmund}. Hence
\begin{equation}
\bigl| s_N(t-d)\,\operatorname{sgn}(a,t)\,e^{-i2\pi f t}\bigr|
\le C\,e^{-a|t|}.
\label{eq:dombound}
\end{equation}
$C e^{-a|t|}$ is integrable on $\mathbb{R}$ and independent of $N$, and
$s_N(t-d)\to s(t-d)$ for almost every $t$. Therefore, by dominated convergence
we can interchange the order of the limit and the integral in \eqref{eq:Pf}.
In addition, since each sum is finite for all $N$ and $t$, we can interchange the order of 
summation and integration. Equation \eqref{eq:Pf} then becomes

\begin{equation}
\begin{split}
\lim_{N\to\infty}\sum_{|n|\le N} e^{-i2\pi n d}\, s_n
   \int_{-\infty}^{\infty}\operatorname{sgn}(a,t)\,e^{-i2\pi (f-n) t}\,dt
&=\sum_{n=-\infty}^{\infty} e^{-i2\pi n d}\, s_n\,\operatorname{SGN}(a,f-n).
\end{split}
\label{eq:sseries2}
\end{equation}

By \eqref{eq:sqcoef} and \eqref{eq:sgn_f} the summands $s_n\,\operatorname{SGN}(a,f-n)$
decay like $1/n^{2}$. This is the value of the signum window: the series
now converges absolutely. Substituting \eqref{eq:sgn_f} evaluated at $f-n$
and the coefficients for the square wave from
\eqref{eq:sqcoef}, we arrive at the Fourier transform of the windowed square wave.

\begin{equation}
P(a,d,f)=\sum_{n=\pm1, \pm3, \dots}e^{-i2\pi n d} \dfrac{8(n-f)}{n(a^2+4\pi^2(n-f)^2)}
\label{eq:sseries3} 
\end{equation}

At $f=0$, \eqref{eq:sseries3} reduces to:

\begin{equation}
\begin{split}
P(a, d, 0) &=  \sum_{n =\pm 1, \pm 3, \dots } \dfrac{8e^{-i2\pi n d}}{a^2+4\pi^2n^2}.
\end{split}
\label{eq:p_a0}
\end{equation}

This is the first major computation of our derivation: an absolutely
convergent series in the delay $d$ and the decay rate $a$.

\subsection*{The time domain.}

We now compute the same quantity in the time domain. That is, 
we find the integral of $p(a, d, t)$ over all $t$, which
is equivalent to the Fourier transform at $f=0$.

\begin{equation}
\begin{split}
\int_{-\infty}^{\infty}p(a, d, t)dt = \int_{-\infty}^{0} (-e^{at})s(t-d) dt + \int_{0}^{\infty}e^{-at}s(t-d)dt
\end{split}
\label{eq:p_even}
\end{equation}

The square wave is odd, so the negative half-line folds onto the
positive one. Using the substitution $u=-t$ for the integral over the $t<0$ region, we can simplify \eqref{eq:p_even} to:

\begin{equation}
\begin{split}
\int_{-\infty}^{0} (-e^{at})s(t-d) dt &= -\int_{0}^{\infty} e^{-au}s(-(u+d)) du \\
&= \int_{0}^{\infty} e^{-at}s(t+d) dt
\end{split}
\label{eq:p_even2}
\end{equation}

In \eqref{eq:p_even2}, we have used the odd property of $s(t)$, where $s(-t)=-s(t)$.
 Breaking the infinite integral into a sum due to the periodicity, we have

\begin{equation}
\begin{split}
\int_{-\infty}^{\infty}p(a, d, t)dt &= \int_{0}^{\infty}e^{-at}\left( s(t+d) + s(t-d)  \right) dt \\
       &= \int_{0}^{1} \sum_{n = 0}^{\infty} e^{-a(t+n)}  \left( s(t+d) + s(t-d) \right) dt \\
\end{split}
\label{eq:p_t2}
\end{equation}

Swapping the order of summation and integration in \eqref{eq:p_t2} is
justified by Fubini's theorem: 
since $|s(t+d)+s(t-d)|\le 2$, the integrand is dominated by the nonnegative
function $2e^{-a(t+n)}$, and consequently for $a>0$,

\begin{equation}
\sum_{n=0}^{\infty}\int_0^1 2e^{-a(t+n)}\,dt
=\frac{2(1-e^{-a})}{a}\sum_{n=0}^{\infty}e^{-an}
=\frac{2}{a}<\infty.
\label{eq:p_t22}
\end{equation}

Therefore, the sum and integral converge absolutely. The
geometric series in \eqref{eq:p_t22} has a closed form sum, and

\begin{equation}
 \sum_{n=0}^{\infty} e^{-an} \int_{0}^{1} e^{-at}
\left(s(t+d)+s(t-d)\right)\,dt
=\frac{1}{1-e^{-a}}
\left(
    \int_{0}^{1}e^{-at}(s(t+d)+s(t-d))\,dt
\right).
\label{eq:p_t}
\end{equation}

We now need to do some elementary algebra and calculus, over a single period of time.
For $0 \leq d\leq0.5$, we can break apart the integrals of \eqref{eq:p_t} as follows.

\begin{equation}
\begin{split}
\int_{0}^{1}e^{-at}s(t+d)dt &=  \int_{0}^{0.5-d} e^{-at} dt - \int_{0.5-d}^{1-d} e^{-at}dt + \int_{1-d}^1 e^{-at} dt \\
\int_{0}^{1}e^{-at}s(t-d)dt &=  \int_{0}^{d} -e^{-at} dt + \int_{d}^{d+0.5} e^{-at}dt - \int_{d+0.5}^1 e^{-at} dt \\
\end{split}
\label{eq:p_t3}
\end{equation}

The equations in \eqref{eq:p_t3} can be combined and simplified:
\begin{equation}
\begin{split}
\int_{0}^{1}e^{-at}s(t+d)dt + \int_{0}^{1}e^{-at}s(t-d)dt 
&= \dfrac{1}{a} \left( -2e^{-a(0.5-d)} + 2e^{-a(1-d)} + 2e^{-ad} - 2e^{-a(d+0.5)} \right) \\
&= \dfrac{2}{a} \left( e^{-ad} (1 - e^{-a/2} ) - e^{ad}e^{-a/2} ( 1 - e^{-a/2} ) \right)\\
\end{split}
\label{eq:p_t4}
\end{equation}

Further factoring and simplifying with hyperbolic sine functions, we reduce \eqref{eq:p_t4} further.

\begin{equation}
\begin{split}
\dfrac{2}{a} \left( e^{-ad} (1 - e^{-a/2} ) - e^{ad}e^{-a/2} ( 1 - e^{-a/2} ) \right) &= \dfrac{2e^{-a/2}}{a}(e^{a/4}-e^{-a/4})(e^{-ad+a/4}-e^{ad-a/4}) \\
&= \dfrac{8e^{-a/2}}{a}\sinh(\dfrac{a}{4})\sinh(\dfrac{a}{4}-ad)
\end{split}
\label{eq:p_t5}
\end{equation}

Since $1-e^{-a}=2e^{-a/2}\sinh(a/2)=4e^{-a/2}\sinh(a/4)\cosh(a/4)$, we can combine \eqref{eq:p_t5} with \eqref{eq:p_t} to obtain

\begin{equation}
\begin{split}
\int_{-\infty}^{\infty}p(a, d, t)dt &= \dfrac{2}{a}\dfrac{\sinh(\dfrac{a(1-4d)}{4})}{ \cosh(\dfrac{a}{4})}
\end{split}
\label{eq:p_t6}
\end{equation}

This is the second major computation of the total integral, which is a closed form without a series.

\subsection*{Equivalence.}

The two calculations must be equivalent. Setting \eqref{eq:p_t6} equal to
\eqref{eq:p_a0}, we get

\begin{equation}
\begin{split}
\sum_{n =\pm 1, \pm 3, \dots } \dfrac{4ae^{-i2\pi n d}}{a^2+4\pi^2n^2} &= \dfrac{\sinh(\dfrac{a(1-4d)}{4})}{ \cosh(\dfrac{a}{4})}.
\end{split}
\label{eq:result1}
\end{equation}

This calculation now leads to the principal result of this article.

\begin{theorem}\label{thm:main}
Let $0 \le d\le \tfrac12$ and $z \in \mathbb{R}$. Then

\begin{equation}
\frac{8z}{\pi^2}
\sum_{n=1,3,\ldots}
\frac{\cos(2\pi nd)}
{n^2+\dfrac{4z^2}{\pi^2}}
=
\frac{\sinh(z(1-4d))}
{\cosh z}.
\label{eq:mainidentity}
\end{equation}

\end{theorem}

\begin{proof}
Equation \eqref{eq:result1} is derived for $a>0$ and $0 \le d \le 1/2$.  It is trivially true for $a=0$, and since it is an odd function of $a$ on both sides, it is true for negative values of $a$ as well, so no restriction on $a$ is required. We next pair the terms $n$ and $-n$ into cosines and substitute $z=a/4$, which
yields \eqref{eq:mainidentity}. 
\end{proof}

Theorem~\ref{thm:main} is a two-parameter identity in $z$ and $d$ which lead down 
distinct mathematical avenues. Freezing the delay at $d=0$ recovers a
classical expansion.

\begin{corollary}\label{cor:tanh}
For $z>0$,
\[
\tanh z =
\frac{8z}{\pi^2}
\sum_{n\ \mathrm{odd}}
\frac{1}{n^2 + 4z^2/\pi^2}.
\]
\end{corollary}

\begin{proof}
With $d=0$, we get the Mittag-Leffler partial-fraction expansion for the hyperbolic tangent.
\end{proof}

Instead of setting $d=0$, we will now send the decay rate to zero, which removes the taper and produces the spectrum of a triangle wave.

\begin{corollary}\label{cor:cosseries}
For $0 \le d \le 1/2$,
\[
\sum_{n = 1,  3, \dots } \dfrac{\cos(2\pi n d)}{n^2} =\dfrac{\pi^2}{8}( 1-4d).
\]
\end{corollary}

\begin{proof}
Starting from Theorem~\ref{thm:main}, we divide by $z$ and take the limit as $z$ goes to zero from above. We interchange the order of the limit and the summation via the dominated convergence theorem, since $1/(n^2+4z^2/\pi^2)\le 1/n^2$ for all $z$ and all $n\ge1$.

\begin{equation}
\begin{split}
\dfrac{8}{\pi^2}\sum_{n = 1,  3, \dots }\lim_{z \to 0^+}  \dfrac{\cos(2\pi n d)}{n^2+\dfrac{4z^2}{\pi^2}} &= \lim_{z \to 0^+}\dfrac{\sinh(z(1-4d))}{z \cosh(z)} \\
\dfrac{8}{\pi^2}\sum_{n = 1,  3, \dots } \dfrac{\cos(2\pi n d)}{n^2} &= \lim_{z \to 0^+} \dfrac{z(1-4d) + z^3(1-4d)^3/6 + \dots}{z(1+z^2/2+\dots)}\\
\sum_{n = 1,  3, \dots } \dfrac{\cos(2\pi n d)}{n^2} &=\dfrac{\pi^2}{8}( 1-4d)
\end{split}
\label{eq:equal}
\end{equation}
\end{proof}

Rearranging~\eqref{eq:equal} gives the Fourier expansion of the first
Euler polynomial, valid for $0 \le d \le \tfrac{1}{2}$:
\begin{equation}
E_1(2d) = 2d - \dfrac{1}{2}=-\frac{4}{\pi^2} \sum_{n = 1, 3, \ldots} \frac{\cos(2\pi n d)}{n^2},
\qquad 0 \le d \le \tfrac{1}{2},
\label{eq:euler1}
\end{equation}
which is \cite[\href{https://dlmf.nist.gov/24.8.E5}{(24.8.5)}]{DLMF} at $n=1$.

From \eqref{eq:equal}, setting $d=0$, we obtain \eqref{eq:lhopit4}.

\begin{equation}
\sum_{n =1,3,5, \dots } \dfrac{1}{n^2}= \dfrac{\pi^2}{8} 
\label{eq:lhopit4}
\end{equation}

Since the sum over the even integers is $1/4$th the sum over all integers, \eqref{eq:lhopit4} leads directly to $\zeta(2)$, $\sum_{n=1}^{\infty} 1/n^2 = \pi^2/6$.

\section{Climbing the Integral Ladder}

The windowed square wave, evaluated in the frequency domain at $f=0$, results in a sum with terms that die off as $1/n^2$. Naturally,
we could next consider windowing the periodic triangle function --- which is the integral of $s(t)$. The windowed triangle function has Fourier series coefficients that fall off 
as $1/n^3$, and following the same process as in Section~2, we would arrive at another 
sum over the odd integers.  However, with no delay added the total integral would now be zero. For this reason, we introduced a delay $d$, which renders the sum nonzero. This observation was the motivation behind introducing the delay in Section~2.

\begin{corollary}\label{cor:sinseries}
For $0 \le d \le 1/2$ and $z>0$,
\[
 \sum_{n =1, 3, \dots } \dfrac{ \sin(2\pi n d)}{n\left(n^2+\dfrac{4z^2}{\pi^2}\right)}
 = \dfrac{\pi^3}{8z^2}\dfrac{\sinh(z(1-2d))\sinh(2zd)}{ \cosh(z)}.
\]
\end{corollary}

\begin{proof}

Instead of redoing the calculations in Section~2 for the triangle wave, we can more easily integrate our result from Theorem~\ref{thm:main} with respect to $d$ and obtain the same result. Starting from \eqref{eq:mainidentity} and integrating from $0$ to $d$, we have

\begin{equation}
   \int_{0}^{d} \left( \dfrac{8z}{\pi^2}\sum_{n = 1, 3, \dots } \dfrac{\cos(2\pi n x)}{n^2+\dfrac{4z^2}{\pi^2}} \right) dx= \int_{0}^{d} \dfrac{\sinh(z(1-4x))}{ \cosh(z)} dx.
\label{eq:integral100}
\end{equation}

We interchange the order of summation and integration (again via Fubini's theorem since the summand is bounded by $1/n^2$ for all $x$), and the calculations proceed as follows. 
\begin{equation}
\left. \dfrac{4z}{\pi^3} \sum_{n =1, 3, \dots } \dfrac{ \sin(2\pi n x)}{n\left(n^2+\dfrac{4z^2}{\pi^2}\right)} \right|_{0}^{d} = \left. \dfrac{-1}{4z \cosh(z)}\cosh(z(1-4x)) \right|_{0}^{d}
\label{eq:integral200}
\end{equation}

\begin{equation}
 \sum_{n =1, 3, \dots } \dfrac{ \sin(2\pi n d)}{n\left(n^2+\dfrac{4z^2}{\pi^2}\right)}
 = \dfrac{\pi^3}{16z^2}\dfrac{\cosh(z)-\cosh(z(1-4d))}{ \cosh(z)}
\label{eq:integral300}
\end{equation}

Using the identity $\cosh(a)-\cosh(b)=2\sinh( (a+b)/2 )\sinh( (a-b)/2)$, we arrive at 
Corollary~\ref{cor:sinseries}. Setting $d=0$ yields the trivial identity.

\end{proof}

\begin{corollary}\label{cor:sind}
For $0 \le d \le 1/2$,
\[
  \sum_{n =1, 3, \dots } \dfrac{ \sin(2\pi n d)}{n^3}
 = \dfrac{\pi^3(d-2d^2)}{4}.
\]
\end{corollary}

\begin{proof}

Taking the limit as $z$ goes to 0 in Corollary~\ref{cor:sinseries}
and employing Maclaurin expansions for the sinh functions, we get

\begin{equation}
\lim_{z \to 0^+} \sum_{n =1, 3, \dots } \dfrac{ \sin(2\pi n d)}{n\left(n^2+\dfrac{4z^2}{\pi^2}\right)}
 = \lim_{z \to 0^+}\dfrac{\pi^3}{8z^2}\dfrac{ [z(1-2d) + O(z^3)][2zd + O(z^3)]}{ 1 + O(z^2)}.
\label{eq:integral500}
\end{equation}

The limit on the left is easily evaluated since the limit and sum can be swapped (again via dominated convergence theorem since the magnitude of the summand is $\le1/n^3$ for all $n\ge 1$). The result is Corollary~\ref{cor:sind}.
\end{proof}

\begin{corollary}\label{cor:sums}
\begin{gather*}
\sum_{n =0 }^{\infty}\dfrac{ (-1)^n}{ (2n+1)^3} =  \dfrac{\pi^3}{32} \\
\frac{1}{1^3} + \frac{0}{3^3} - \frac{1}{5^3} + \frac{1}{7^3} + \frac{0}{9^3} - \frac{1}{11^3} + \dots{}=  \dfrac{\pi^3}{18\sqrt{3}} \\
1+\frac{1}{3^3}-\frac{1}{5^3}-\frac{1}{7^3}
+\frac{1}{9^3}+\frac{1}{11^3}-\cdots
 =  \dfrac{3\pi^3}{64\sqrt{2}}
\end{gather*}
\end{corollary}

\begin{proof}

Evaluating Corollary~\ref{cor:sind} at $d=1/4$ gives the first sum, equivalent to the Dirichlet beta function evaluated at 3, $\beta(3)$. The second sum is found from applying $d=1/6$ and is equivalent to the Dirichlet L-function at 3 for the odd character mod 6. The last sum is found from applying $d=1/8$, and the result is equivalent to the Dirichlet L-function at 3 for the odd primitive character mod 8.

\end{proof}

\begin{corollary}\label{cor:cosd}
For $0\le d \le 1/2$,
\begin{gather*}
\sum_{n =1, 3, \dots } \dfrac{ \cos(2\pi n d)}{n^4}  = \pi^4(\dfrac{1}{96}-\dfrac{d^2}{4}+\dfrac{d^3}{3}) 
\end{gather*}
\end{corollary}

\begin{proof}

Beginning with Corollary~\ref{cor:sind}, we can again integrate from $d$ to $1/4$:

\begin{equation}
\int_{d}^{1/4}\sum_{n =1, 3, \dots } \dfrac{ \sin(2\pi n x)}{ n^3} dx= \int_{d}^{1/4}\dfrac{ \pi^3 (x-2x^2) }{4} dx.
\label{eq:integral3}
\end{equation}

As before the sum and integral can be interchanged via dominated convergence, and we arrive at

\begin{equation}
\begin{split}
\left. \sum_{n =1, 3, \dots } \dfrac{ \cos(2\pi n x)}{-2\pi n^4} \right|_{d}^{1/4} = \left. \dfrac{\pi^3(\dfrac{x^2}{2}-\dfrac{2x^3}{3})}{4} \right|_{d}^{1/4}.
\end{split}
\label{eq:integral4}
\end{equation}

Equation \eqref{eq:integral4} simplifies to our desired result.
\end{proof}

We can let $d=0$ in Corollary~\ref{cor:cosd}, and get $\sum_{n =1, 3, \dots } 1/n^4  = \pi^4/96$. Since the sum over the even integers of $1/n^4$ is $1/16$th the sum over all integers, we arrive at $\zeta(4)$:

\begin{equation}
\sum_{n =1 }^{\infty} \dfrac{ 1}{ n^4} = \dfrac{ \pi^4}{90}.
\label{eq:zeta4} 
\end{equation}

We've shown how integration of a cosine series produces a sine series with an additional 
$1/n$ factor, and integration again gets back to a cosine series, again divided by $n$. 
We can continue this process indefinitely, and can prove by induction this method 
would yield $\zeta(2k)$ for all positive $k$.

\begin{theorem}\label{thm:zeta2k}
For $k \ge 1$, 

\begin{equation}
\zeta(2k)=\frac{(2\pi)^{2k}Q_{2k-1}(0)}{2^{2k}-1},
\label{eq:zetazeta} 
\end{equation}

where $Q_m$ is a polynomial of degree $\le m$ with rational coefficients, defined recursively by

\begin{equation}
Q_1(d) = \dfrac{1-4d}{8} ,
\label{eq:qdef1} 
\end{equation}

\begin{equation}
Q_{2m}(d) = \int_{0}^{d}2Q_{2m-1}(x)dx,
\label{eq:qdef2k} 
\end{equation}

\begin{equation}
Q_{2m+1}(d)=\int_{d}^{1/4}2Q_{2m}(x)dx.
\label{eq:qdef2kp1} 
\end{equation}
\end{theorem}

\begin{proof}
Assume for some $k\ge1$ that:

\begin{equation}
\sum_{n =1, 3, \dots } \dfrac{ \cos(2\pi n d)}{n^{2k}}  = \pi^{2k}Q_{2k-1}(d).
\label{eq:induction1}
\end{equation}

Proceeding as in Corollary~\ref{cor:sinseries}, we can integrate both sides of \eqref{eq:induction1} from $0$ to $d$, replacing $d$ with the dummy variable $x$ in the integrand. The integral of the left hand side of \eqref{eq:induction1} follows as in \eqref{eq:integral100}, again with integration and summation swapping order via Fubini's theorem as the summand is bounded by $1/n^{2k}$.  The right hand side is the integral of a polynomial of degree $2k-1$ evaluated at $d$ and $0$, thus remaining a polynomial 
with rational coefficients but of degree at most $2k$. For convenience, we will absorb the $1/2$ term from the left hand side integral, so that we have the definition given in \eqref{eq:qdef2k}.

The $\pi$ factor moves to the right side, and the result becomes

\begin{equation}
\sum_{n =1, 3, \dots } \dfrac{ \sin(2\pi n d)}{n^{2k+1}}  = \pi^{2k+1}Q_{2k}(d).
\label{eq:induction2}
\end{equation}

Similarly, we repeat the process in Corollary~\ref{cor:cosd} and integrate \eqref{eq:induction2} from $d$ to $1/4$. Again, the left hand side proceeds as in 
\eqref{eq:integral3}, producing a cosine series. The upper integral limit of $1/4$ 
causes the cosine term to vanish, since $\cos(2\pi n/4)=0$ for all odd $n$. On the right side, this value also keeps the results constrained to rational values. The polynomial is given by
\eqref{eq:qdef2kp1}, and the integration of the sine series yields

\begin{equation}
\sum_{n =1, 3, \dots } \dfrac{ \cos(2\pi n d)}{n^{2k+2}}  = \pi^{2k+2}Q_{2k+1}(d).
\label{eq:induction3}
\end{equation}

Since we have $k=1$ established from Corollary~\ref{cor:cosseries} and the inductive step is true for arbitrary $k\ge 1$, the 
result in \eqref{eq:induction3} extends to all positive integers. Consequently, $d=0$ gives the sum over the odd integers, 

\begin{equation}
\sum_{n =1, 3, \dots } \dfrac{ 1}{n^{2k}}  = \pi^{2k}Q_{2k-1}(0).
\label{eq:induction4}
\end{equation}

We also see that the sum of $1/n^{2k}$ over the even integers equals $\zeta(2k)/2^{2k}$. Hence, a little algebra produces \eqref{eq:zetazeta} and the proof is complete.
\end{proof}

\section{Remarks}
 
The derivations are successive iterations of a single construction: each result
followed from the previous by termwise integration. The cosine
series of Theorem~\ref{thm:main} is integrated to the sine series in Corollary~\ref{cor:sinseries},
and then to a fourth-order cosine series in Corollary~\ref{cor:cosd}, and Theorem~\ref{thm:zeta2k} continues the ladder to the even values $\zeta(2k)$.
The odd-exponent rungs vanish at $d=0$ and
yield closed forms only at the shifts $d=1/4,1/6,1/8$, where the odd Dirichlet characters are 
reproduced. The odd-exponent values from this
method create series like $\beta(3)$, but unfortunately 
never $\zeta(3)$ because this would require the sum coefficients to all be +1, but this cannot 
be done with a $\sin(2\pi n d)$ in the summand for any $d$.
Both the Basel
sum and the tanh expansion are special cases of a single
identity. Setting $d=0$ gives the Mittag-Leffler expansion
of Corollary~\ref{cor:tanh}, while $z\to0^{+}$ gives the cosine
series~\eqref{eq:equal} and hence $\zeta(2)$, so the decaying-signum window of the square wave is the common source of both.

\noindent\textbf{Summary}\hspace{2em}
  We window a delayed periodic square wave with a decaying signum function and evaluate the product at zero frequency in both the time and frequency domains. The window renders every series that appears absolutely convergent, so the argument requires only elementary integrals and a limit. We
  obtain new unified derivations for the Mittag-Leffler partial-fraction expansion for the hyperbolic tangent, the Fourier expansion 
  of the first Euler polynomial $E_1$, and the Basel sum $\zeta(2)$. The same calculation, followed by termwise integration, yields the Dirichlet beta value $\beta(3)=\pi^{3}/32$, and the values at $s=3$ of two other Dirichlet $L$-functions. Termwise integration of this sine series produces $\zeta(4)$, and an induction continues the process to $\zeta(2k)$ for every $k\ge1$.

\end{document}